\documentclass[12pt]{article}
\usepackage{times}
\usepackage[left=1in,top=1in,right=1in]{geometry}
\usepackage{cite}
\usepackage{amsmath,amsthm,amssymb}
\usepackage{color}
\usepackage{latexsym}
\usepackage{cite}
\usepackage{amsmath,amssymb,amsthm}
\title{On Complex Sasakian manifolds }
\author{Aysel Turgut Vanli$ ^1 $, \.Inan \"Unal$ ^2 $ and   Keziban Avcu$ ^1$ \\
	$ ^1 $Gazi University and $ ^2 $Munzur University }
\date{}

\theoremstyle{plain}

\newtheorem{theorem}{Theorem}
\theoremstyle{plain}

\newtheorem{corollary}{Corollary}

\newtheorem{definition}{Definition}
\newtheorem{example}{Example}

\newtheorem{lemma}{Lemma}

\newcommand{\NCM}{normal{\color{white}$ \theta $}complex{\color{white}$ \theta $}contact{\color{white}$ \theta $}metric{\color{white}$ \theta $}manifold}

\begin{document}
\maketitle
\date{}


\noindent \textbf{Abstract:} In this article we study a class of \NCM\ which is
called a complex Sasakian manifold. This kind of manifold has a globally defined complex contact form and normal complex contact structure. We give the definition of a complex Sasakian manifold by consider the real case and we present general properties. Also we obtain some useful curvature relations. Finally we examine flatness conditions for general curvature tensor B. 
{$^1 $avanli@gazi.edu.tr, $^1 $keziban.uysal2812@gmail.com , $^2 $inanunal@munzur.edu.tr}

\section{ Introduction}

Complex contact manifolds are still has many open problems. The importance of this subject is not only the complex version of real contact manifolds also one can find many important informations about complex manifolds, K\"ahler manifolds. In addition there are some applications in theoretical physics  \cite{kholodenko2013applications}. Complex contact manifolds has an old history same as real contact manifolds, researchers could not give their attention to the subject. When we look at the 1980s there are very important improvements in the Riemannian geometry of complex contact manifolds. Ishihara and Konishi constructed tensorial relations for a complex almost contact structure and they also presented normality \cite{ishihara1979real,ishihara1980complex,ishihara1982complex}. The Riemann geometry of complex almost contact metric manifolds could be divided into three notions;
\begin{enumerate}
	\item \textbf{\textit{IK-normal complex contact metric manifolds}} : A complex contact manifold has a normal contact structure in the sense of Ishihara-Konishi. This type of manifolds were studied in \cite{ishihara1979real,ishihara1980complex,ishihara1982complex,imada2014construction,imada2015complex,turgut2018h}.
	\item \textbf{\textit{Normal complex contact metric manifolds}}: A complex contact manifold has a normal contact structure in the sense of Korkmaz.  This type of manifolds were studied in \cite{korkmaz2000,Korkmaz2003,blair2006corrected,blair2009special,blair2011bochner,blair2012homogeneity,blair2012symmetry,vanli2015curvature,turgut2017conformal,vanli2017complex}.
	\item \textbf{\textit{Complex Sasakian manifolds}}: A complex contact manifold with a globally complex contact form and has a normal contact structure in the sense of Korkmaz. This type of manifolds were studied in \cite{foreman2000complex,fetcu2006harmonic,fetcuadapted}.
\end{enumerate}
In this work we study on third type of these manifolds. Firstly we adopted the definition of a complex Sasakian manifold by consider the definition of a real Sasakian manifold. Later we give some fundamental equations and we obtain curvature properties. Finally we examine some flatness conditions. We use a general tensor which is defined in \cite{shaikh2018some} and is called by $ B- $tensor. We prove that a complex Sasakian manifold could not be $ B- $flat.

\section{Preliminaries}
In this section we give fundamental facts on complex contact manifolds. For details reader could be read \cite{foreman1996variational,korkmaz2000,blair2010riemannian}.
\begin{definition}
	Let $N$ be a complex manifold of odd complex dimension $2p+1$ covered by an
open covering $\mathcal{C}=\left\{ \mathcal{A}_{i}\right\} $ consisting of
coordinate neighborhoods. If there is a holomorphic $1$-form $\eta _{i}$
on each $\mathcal{A}_{i}\in \mathcal{C}$ in such a way that for any $%
\mathcal{A}_{i},\mathcal{A}_{j}\in \mathcal{C}$ and for a holomorphic function $f_{ij}$ on $\mathcal{A}_{i}\cap \mathcal{A}_{j}\neq \emptyset $
\begin{equation*}
\eta _{i}\wedge (d\eta _{j})^{p}\neq 0\text{ in }\mathcal{A}_{i}\text{, }
\end{equation*}%
\begin{equation*}
\eta _{i}=f_{ij}\eta _{j},\text{ }\mathcal{A}_{i}\cap \mathcal{A}%
_{j}\neq \emptyset ,
\end{equation*}%
then the set $\left\{ \left( \eta _{i},\mathcal{A}%
_{i}\right) \mid \mathcal{A}_{i}\in \mathcal{C}\right\} $ of local structures is called complex contact structure and with this structure $N$ is called a complex contact manifold.
\end{definition}

The complex contact structure determines a non-integrable distribution $H_i$ by the equation $\eta_i =0$ such as 
\begin{equation*}
H_{i}=\{X_{P}:\eta_{i}(X_{P})=0, X_{P}\in T_{P}N\}.
\end{equation*}
and a holomorphic vector field $ \xi_i $ 
is defined by  

\begin{equation*}
\eta_i(\xi_i)=1
\end{equation*}
and a complex line bundle is defined by $ E_i=Span\{\xi_i\} $.\\
\qquad Let $ T^c(N) $ be complexified of tangent bundle of $ (N,J,\eta_i) $ and let define vector fields 
\begin{eqnarray*}
	U_i=\xi_i+\bar{\xi_i}\ \ \ \ \ \ \ V_i=-i(\xi_i+\bar{\xi_i})
\end{eqnarray*}
and 1-forms
\begin{eqnarray*}
	u_i=\frac{1}{2}(\eta_i+\bar{\eta_i}) \ \ \ \ \ \ v_i=\frac{1}{2}i(\eta_i-\bar{\eta_i}).
\end{eqnarray*}
Therefore we get 
\begin{enumerate}
	\item $ V_i=-JU_i $ and $ v_i=u_i\circ J $
	\item $ U_i=JV_i $ and $ u_i=-v_i\circ J $
	\item $ u_i(U_i)=v_i(V_i)=1 $ and $ u_i(V_i)=v_i(V_i)=0 $.
\end{enumerate}
The complexified $ H_i $ and $ E_i $ is defined by 
\begin{eqnarray*}
{	H_i}^c&=&\{W\in T^c(N)\arrowvert u(W)=v(W)=0\}\\
{	E_i}^c&=&Span\{U,V\}. 
\end{eqnarray*}
We use notation $ \mathcal{H} $ and $ \mathcal{V} $ for the union of $ {	H_i}^c $ and $ E_i^c $ respectively. $ \mathcal{H} $ is called horizontal distribution and $ \mathcal{V} $ is called vertical distribution 
and we can write 
\begin{equation} \label{TN=H+V}
TN=\mathcal{H}\oplus\mathcal{V}.
\end{equation}
 Ishihara and Konishi \cite{ishihara1980complex} proved that $N$ admits always a complex contact structure of $ C^{\infty} $.
\begin{definition} \label{complexalmostscontact}
	Let $N$ be a complex manifold with complex structure $ J $, Hermitian metric $ g $ and $\mathcal{C=}\left\{ 
	\mathcal{A}_{i}\right\} $ be open covering of $N$ with coordinate
	neighbourhoods $\{\mathcal{A}_{i}\mathcal{\}}.$ If $N$ satisfies the
	following two conditions then it is called a \textit{complex almost contact
		metric manifold}:
	
	1. In each $\mathcal{A}_{i}$ there exist $1$-forms $u_{i}$ and $%
	v_{i}=u_{i}\circ J$, with dual vector fields $U_{i}$ and $V_{i}=-JU_{i}$ and 
	$(1,1)$ tensor fields $G_{i}$ and $H_{i}=G_{i}J$ such that 
	\begin{equation} \label{G^2veH^2}
	H_{i}^{2}=G_{i}^{2}=-I+u_{i}\otimes U_{i}+v_{i}\otimes V_{i}
	\end{equation}%
	\begin{equation*} \label{H=GJ}
	G_{i}J=-JG_{i},\quad GU_{i}=0,\quad
	\end{equation*}%
	\begin{equation*} \label{g(GX,Y)=-g(X,GY)}
	g(X,G_{i}Y)=-g(G_{i}X,Y).
	\end{equation*}%
	2.On $\mathcal{A}_{i}\cap \mathcal{A}_{j}\neq \emptyset $ we have 
	\begin{eqnarray*}
		u_{j} &=&au_{i}-bv_{i},\quad v_{j}=bu_{i}+av_{i},\; \\
		G_{j} &=&aG_{i}-bH_{i},\quad H_{j}=bG_{i}+aH_{i}
	\end{eqnarray*}%
	where $a$ and $b$ are functions on $\mathcal{U}_{i}\cap \mathcal{U}_{j}$
	with $a^{2}+b^{2}=1$ \cite{ishihara1979real,ishihara1980complex}.
\end{definition}
By direct computation we have 
\begin{eqnarray} 
H_{i}G_{i} &=&-G_{i}H_{i}=J_{i}+u_{i}\otimes V_{i}-v_{i} \otimes U_{i} \label{HG=-GH} \label{HG,GH}\\
J_{i}H_{i} &=&-H_{i}J_{i}=G_{i}  \notag \\
G_{i}U_{i} &=&H_{i}U_{i}=H_{i}V_{i}=0  \notag \\
u_{i}G_{i} &=&v_{i}G_{i}=u_{i}H_{i}=v_{i}H_{i}=0 \notag \\
J_{i}V_{i} &=&U_{i}, \ g(U_{i},V_{i})=0 \notag\\
g(H_{i}X,Y) &=&-g(X,H_{i}Y) \notag.
\end{eqnarray}

By the local contact form $\eta $ is $u-iv$ to within a
non-vanishing complex-valued function multiple and the local fields $G$ and $%
\ H$ are related to $du$ and $dv$ by 
\begin{eqnarray*}
	du(X,Y) &=&g(X,GY)+(\sigma \wedge v)(X,Y),~~~ \\
	dv(X,Y) &=&g(X,HY)-(\sigma \wedge u)(X,Y)
\end{eqnarray*}
where $\sigma (X)=g(\nabla _{X}U,V)$, $\nabla $ being the Levi-Civita
connection of $g$ \cite{ishihara1980complex}. $ \sigma $ is called IK-connection \cite{foreman1996variational}.

Ishihara and Konishi \cite{ishihara1980complex} study on normality of complex   almost contact metric manifolds. They defined local tensors
\begin{eqnarray*}
	S(X,Y) &=&[G,G](X,Y)+2g(X,GY)U-2g(X,HY)V \\
	&&+2(v(Y)HX-v(X)HY)+\sigma (GY)HX \\
	&&-\sigma (GX)HY+\sigma (X)GHY-\sigma (Y)GHX,
\end{eqnarray*}
\begin{eqnarray*}
	T(X,Y) &=&[H,H](X,Y)-2g(X,GY)U+2g(X,HY)V \\
	&&+2(u(Y)GX-u(X)GY)+\sigma (HX)GY \\
	&&-\sigma (HY)GX+\sigma (X)GHY-\sigma (Y)GHX
\end{eqnarray*}
where
\begin{equation*}
\lbrack G,G](X,Y)=(\nabla _{GX}G)Y-(\nabla _{GY}G)X-G(\nabla
_{X}G)Y+G(\nabla _{Y}G)X
\end{equation*}%
is the Nijenhuis torsion of $G$. 
\begin{definition}
\cite{ishihara1980complex} A complex{\color{white}$ \eta $}almost{\color{white}$ \eta $}contact{\color{white}$ \eta $}metric{\color{white}$ \eta $}manifold is called IK-Normal{\color{white}$ \eta $} if $ S=T=0 $. 
\end{definition}
An IK-Normal manifold has K\"ahler structure. In other words a non-K\"ahler complex almost contact metric manifold is not to be IK-Normal such Iwasawa manifold. Iwasawa manifold is not K\"ahler and  it is compact manifold which has symplectic structure. Also Baikousis et al. \cite{baikoussis1998holomorphic} obtained complex almost contact structure on Iwasawa manifold. Korkmaz gave a weaker definition for normality and Iwasawa manifold is normal in the sense of this definition. 
\begin{definition} [Korkmaz's Definition,  \cite{korkmaz2000} ]
	A complex almost contact metric manifold is called normal if
	
	$\qquad S(X,Y)=T(X,Y)=0$ \ for all $X,Y$ in $\mathcal{H},$ and $\ $
	
	$\qquad S\left( X,U\right) =T\left( X,V\right) =0$ for all $X.$
\end{definition} 

Also for arbitrary vector fields $ X$ on $N$ we have \cite{ishihara1980complex,foreman1996variational,korkmaz2000} 
\begin{align}
\nabla _{X}U &=-GX+\sigma (X)V,~~  \label{nablaXU} ,
\ \ \ \nabla _{X}V =-HX-\sigma (X)U,~~  \\
\nabla _{U}U &=\sigma (U)V,~~~\nabla _{U}V=-\sigma (U)U  \label{nablaUU} , \ \ 
\nabla _{V}U =\sigma (V)V,~~~~\nabla _{V}V=-\sigma (V)U.~~  
\end{align}

\begin{theorem}
	A complex almost
	contact metric manifold is normal if and only
	if the covariant derivative of \ $G$ and $H$ have the following
	forms: 
	\begin{eqnarray}
	(\nabla _{X}G)Y &=&\sigma (X)HY-2v(X)JY-u\left( Y\right) X
	\label{Yeni normallik G} \\
	&&-v(Y)JX+v(X)\left( 2JY_{0}-\left( \nabla _{U}J\right) GY_{0}\right)  \notag
	\\
	&&+g(X,Y)U+g(JX,Y)V  \notag \\
	&&-d\sigma (U,V)v(X)\left( u(Y)V-v(Y)U\right)  \notag
	\end{eqnarray}
	and%
	\begin{eqnarray}
	(\nabla _{X}H)Y &=&-\sigma (X)GY+2u(X)JY+u(Y)JX  \label{Yeni normallik H} \\
	&&-v(Y)X+u(X)\left( -2JY_{0}-\left( \nabla _{U}J\right) GY_{0}\right)  \notag
	\\
	&&-g(JX,Y)U+g(X,Y)V  \notag \\
	&&+d\sigma (U,V)u(X)\left( u(Y)V-v(Y)U\right)  \notag
	\end{eqnarray}
	where $X=X_{0}+u(X)U+v(X)V$ and $Y=Y_{0}+u(Y)U+v(Y)V,X_{0},Y_{0}\in $ $%
	\mathcal{H}$ \cite{vanli2015curvature}.
\end{theorem}
From this theorem on a normal complex contact metric manifold we have

\begin{eqnarray*}
	(\nabla _{X}J)Y &=&-2u\left( X\right) HY+2v(X)GY+u(X)\left( 2HY_{0}+\left(
	\nabla _{U}J\right) Y_{0}\right) \\
	&&+v(X)\left( -2GY_{0}+\left( \nabla _{U}J\right) JY_{0}\right) .
\end{eqnarray*}

As we have seen, there are two normality notions for a \NCM. The other kind of \NCM s is complex Sasakian manifold. This type of manifolds are normal due to Korkmaz's definition  and they were studied in \cite{foreman2000complex,fetcu2006harmonic,fetcuadapted}. The fundamental difference of this type from others is globally definition of complex contact form. Kobayashi \cite{kobayashi1959remarks} proved that a complex contact form is globally defined if and only if first Chern class of $ N $ vanishes.  Also Foreman obtained following result; 
\begin{lemma}
	Let $ (N,\eta_i) $ be a complex contact manifold. If $ \eta_i $ is globally defined i.e $ f_{ij}=1 $ then $ \sigma=0 $ \cite{foreman2000complex}. 
\end{lemma}
 A complex Sasakin manifold is defined as follow; 
\begin{definition}
	Let $\left(N,G,H,J,U,V,u,v,g\right) $ be a normal complex contact metric
	manifold and $\eta =u-iv$ is globally defined. If fundemental 2- forms $%
	\widetilde{G}$ and $\widetilde{H}$ is defined by 
	\begin{equation*}
	\widetilde{G}\left( X,Y\right) =du(X,Y)\text{ and \ }\widetilde{H}\left(
	X,Y\right) =dv\left( X,Y\right)
	\end{equation*}%
	then $ N $ is called a complex Sasakian manifold, where $X,Y$ are vector fields on 
	$N$.
\end{definition}
Thus we get following result;
\begin{theorem}
	Let $ N $ be a \NCM. Then $N$ is complex Sasakian if and only
	if\bigskip 
	\begin{eqnarray*}
	(\nabla _{X}G)Y&=&-2v(X)HGY-u(Y)X-v(Y)JX+g(X,Y)U+g(JX,Y)V \\
	(\nabla _{X}H)Y &=&-2u(X)HGY+u(Y)JX-v(Y)X-g(JX,Y)U+g(X,Y)V.
	\end{eqnarray*}
\end{theorem}
This result was also given by  Ishihara-Konishi\cite{ishihara1980complex}. But should have been considered that $ N $ is not K\"ahler in here. So,we have 
\begin{eqnarray*}
	(\nabla _{X}J)Y &=&-2u(X)HY+2v(X)GY.
\end{eqnarray*}
From (\ref{nablaXU} )  on a complex Sasakin manifold we get 
\begin{eqnarray}
\nabla _{X}U =-GX,~~ \ \ \ \nabla _{X}V =-HX. \label{NablaX-U}
\end{eqnarray}
On the other hand we have 
\begin{corollary}
	An IK-normal complex contact metric manifold could not be complex Sasakian \cite{turgut2018h}. 
\end{corollary} 
 This result support that the geometry of complex Sasakian manifolds has some different properties. 
\section{Curvature Properties of Complex Sasakian Manifolds}
In the Riemannian geometry of contact manifolds curvature properties have an important position. We use these relations for future works. The curvature relations of a complex almost contact metric manifolds were given in \cite{foreman1996variational}, an IK-Normal manifold were given in \cite{ishihara1980complex} and a normal complex contact metric manifold were given in \cite{korkmaz2000,vanli2015curvature}. In this section by taking advantage from these curvature properties, we present curvature relations for complex Sasakian manifold.
Let $ N $ be a complex Sasakian manifold. Then for $ X,Y \in \Gamma(TN) $ we have, 
\begin{align}
R\left( U,V\right) V&=R\left( V,U\right) U=0 \label{SasakianR(UVV)} \\
R(X,U)U&=X+u(X)U+v(X)V	\label{SasakianR(XUU)} \\
R(X,V)V&=X-u(X)U-v(X)V	\label{SasakianR(XVV)} \\
R(X,U)V&=-3JX-3u(X)V+3v(X)U \label{SasakianR(XUV)}\\
R(X,V)U&=0 \label{SasakianR(XVU)}\\
R(X,Y)U&=v(X)JY-v(Y)JX+2v(X)u(Y)V-2v(Y)u(X)V  \label{SasakianR(XYU)}\\
&+u(Y)X-u(X)Y-2g(JX,Y)V \notag \\
R(X,Y)V&=3u(X)JY-3u(Y)JX-2u(X)v(Y)U+2u(Y)v(X)U \label{SasakianR(XYV)} \\
&+v(Y)X-v(X)Y+2g(JX,Y)U \notag \\
R(U,V)X&=JX+u(X)V-v(X)U\label{SasakianR(UVX)} \\	R(X,U)Y&=-2v(Y)v(X)U+2u(Y)v(X)V-g(Y,X)U \label{SasakianR(XUY)}\\
&+u(Y)X+g(JY,X)V  \notag \\
R(X,V)Y&=3u(Y)JX+2u(Y)u(X)V+3g(JY,X)U\label{SasakianR(XVY)} \\&-2v(Y)u(X)U-g(Y,X)V+v(Y)X-2u(X)JY \notag .
\end{align}
 For $X,Y,Z,W$ horizontal vector fields we have \cite{vanli2015curvature} 
\begin{equation}\label{r(gx,gy,gz,gw)}
g(R(GX,GY)GZ,GW)=g(R(HX,HY)HZ,HW)=g(R(X,Y)Z,W)
\end{equation}
and we have 
	\begin{align*}
g(R(X,GX)GX,X)&+g(R(X,HX)HX,X)+g(R(X,JX)JX,X)=-6g\left( X,X\right)\\
g(R(X,GX)Y,GY) &=g(R(X,Y)X,Y)+g(R(X,GY)X,GY)-2g(GX,Y)^{2} \\	
&-4g(HX,Y)^{2}-2g(X,Y)^{2} +2g(X,X)g(Y,Y)-4g(JX,Y)^{2}\\
g(R(X,HX)Y,HY) &=g(R(X,Y)X,Y)+g(R(X,HY)X,HY)-2g(HX,Y)^{2} \\
&-4g(GX,Y)^{2}-2g(X,Y)^{2}+2g(X,X)g(Y,Y)-4g(JX,Y)^{2} \\
g(R(X,HX)JX,GX)&=-g(R(X,HX)HX,X)-4g(X,X)^{2}\\
g(R(X,JX)HX,GX)&=g(R(X,JX)JX,X)-2g(X,X)^{2}.
\end{align*}
\begin{align*}
g(R(GX,HX)HX,GX)&=g(R(X,JX)JX,X) \\
g(R(GX,JX)JX,GX)&=g(R(X,HX)HX,X)\\
g(R(JX,JY)JY,JX)&=g(R(X,Y)Y,X)\\
g(R(X,Y)JX,JY)&=g(R(X,Y)Y,X)+4g(X,GY)^{2}+4g(X,HY)^{2}\\
g(R(Y,JX)JX,Y)&=g(R(X,JY)JY,X) \\
g(R(X,JY)JX,Y)&=g(R(X,JY)JY,X)+4g(X,HY)^{2}+4g(X,GY)^{2} \\
g(R(X,JX)JY,Y)&=-g(R(JX,JY)X,Y)-g(R(JY,X)JX,Y) \\
g(R(X,JX)JY,Y)&=g(R(X,Y)Y,X)+g(R(X,JY)JY,X)\\
&+8\left(g(X,GY)^{2}+g(X,HY)^{2}\right). 
\end{align*}
We have an nice relation as follow;
\begin{corollary}
	For a unit horizontal vector $ X $ on $N$ we have
	\begin{equation}
	k\left( X,GX\right) +k\left( X,HX\right) +k\left( X,JX\right) =6. \label{sectionalrelation}
	\end{equation}
\end{corollary}
This relation also valid for an IK-normal complex contact metric manifold \cite{imada2014construction}. 
\begin{corollary}
	Let $N$ be a complex Sasakian manifold and $ X $ be a unit  horizontal vector field on $N$. Then for the sectional curvature $ k $ we have 
	\begin{equation*}
	k(U,V)=0\, \ \ \text{and} \ \ \ k(X,U)=1.
	\end{equation*}
\end{corollary}
Turgut Vanl\i\  and Unal \cite{vanli2015curvature} presented properties of Ricci curvature tensor
of a normal{\color{white}$ \theta $}complex{\color{white}$ \theta $}contact{\color{white}$ \theta $} metric manifold. For complex Sasakian case 
we have following relations; 
 
	\begin{eqnarray} \label{Ricciler}
		\rho (U,U) &=&\rho (V,V)=4p,\text{ \ }\rho (U,V)=0\\
		\rho (X,U) &=&4pu(X)  \notag,\ 
		\rho \left( X,V\right) =4pv(X)  \notag\\
		\rho (X,Y) &=&\rho (GX,GY)+4p\left( u(X)u(Y)+v(X)v(Y)\right) 
		\label{Ricci(X,Y)=Ricci(GX,GY)+} \notag\\
		\rho (X,Y) &=&\rho (HX,HY)+4p\left( u(X)u(Y)+v(X)v(Y)\right)  \notag
	\end{eqnarray}
where $ X,Y \in \Gamma(TN) $.\par 
The sectional curvature of Riemann manifold give us important information about the geometry of the manifold. In complex manifold we have holomorphic sectional curvature which is the curvature of section is spanned by $ X $  and $ JX $. Similarly in contact manifold we have $ \phi- $sectional curvature which is the curvature of section is spanned by $ X $  and $ \phi X $ \cite{blair2010riemannian}. For complex contact case we have sectional curvature, holomorphic sectional curvature and $ \mathcal{GH}- $sectional curvature which was given by Korkmaz \cite{korkmaz2000} as below: 
\begin{definition}
	\cite{korkmaz2000} Let $ N $ be a \NCM. $X$ be an unit horizontal vector field  on $N$ and $a^{2}+b^{2}=1$.  A $\mathcal{GH-}$section is a plane which is spanned by $X$ and $Y=aGX+bHX$ and the sectional curvature of this plane is called $\mathcal{GH-}$\textit{sectional curvature } is defined by $%
	\mathcal{GH}_{a,b}\left( X\right) =k\left( X,aGX+bHX\right) ,$ where $%
	k(X,Y)$\ is the sectional curvature of the plane section spanned by $X$ and $Y$.
\end{definition}
$\mathcal{GH-}$\textit{sectional curvature } is denoted by $ \mathcal{GH}_{a,b}  $ and we assume that it does not depend the choice of $a$ and $b$. So we
will use $\mathcal{GH}\left( X\right) $ notation. \par

Let $ N $ be a complex Sasakian manifold. Since the complex contact form is globally defined $ \mathcal{GH}_{a,b}  $ does not depend the choice of $a$ and $b$ in naturally. Also we have \cite{korkmaz2000}
\begin{equation}
k(X,JX)=\mathcal{GH}\left( X\right) +3.  \label{k(X,JX)=GH(X)+3}
\end{equation}
Thus from (\ref{sectionalrelation}) we obtain
	\begin{equation}
k\left( X,GX\right) +k\left( X,HX\right) +\mathcal{GH}\left( X\right) =3. 
\end{equation}
\begin{example}
	The well known example of normal complex contact metric manifold complex Heisenberg group. A globally defined complex contact form and complex almost contact structure on complex Heisenberg group was given by Baikousis et al. \cite{baikoussis1998holomorphic}. Korkmaz obtained normality of complex Heisenberg group. Thus a complex Heisenberg group is an example of complex Sasakian manifolds. For details about complex Heisenberg group see \cite{blair2006corrected,foreman2000complex,korkmaz2000,vanli2015curvature}.   
\end{example}
\begin{example}
	An other example of complex Sasakian manifolds were given by Foreman \cite{foreman2000complex}. Foreman obtained an example from hyperk\"ahler manifold. For details see \cite{foreman2000complex}. 
\end{example}
\section{Flatness on complex Sasakian manifolds}
A real Sasakian manifold can not be flat, i.e its Riemannian curavture could not be zero identically \cite{de2009complex}. We obtain same results for complex Sasakin manifolds. 
\begin{theorem}
	A complex Sasakian manifold can not be flat.
\end{theorem}

\begin{proof}
	Suppose that a complex Sasakian manifold is flat. Then $R(X,Y)Z=0$
	and from that $\rho \left( X,Y\right) =0.$ On the other hand from (\ref{Ricciler}) we have 
	\begin{equation*}
	4nu(X)=0.
	\end{equation*}
	This is not possible. So the manifold can not be flat.
\end{proof}
 Shaikh and Kundu \cite{shaikh2018some} generalized curvature tensors and gave the definition of B-tensor. They proved some equivalence relations between most of certain curvature conditions. 
\begin{definition}
	$N$ be a complex Sasakian manifold . $ (0, 4) $ tensor $ B $ of $N$ is given by
	\begin{eqnarray} \label{B-TENSOR}
	B(X,Y,Z,W) &=& a_{0}R(X,Y,Z,W) + a_{1}R(X,Z,Y,W)\\
	&+& a_{2}\rho(Y,Z)g(X,W)+a_{3}\rho(X,Z)g(Y,W)\notag\\
	&+&a_{4}\rho(X,Y)g(Z,W)+a_{5}\rho(X,W)g(Y,Z)\notag \\
	&+&a_{6}\rho(Y,W)g(X,Z)+a_{7}\rho(Z,W)g(X,Y) \notag \\
	&+& \tau\{a_{8}g(X,W)g(Y,Z)+a_{9}g(X,Z)g(Y,W) \notag\\
	&+&_{10}g(X,Y)g(Z,W)\} \notag
	\end{eqnarray}
	where $ a_{i} $'s are scalars on $  M $ and $ X,Y,Z,W \in \Gamma(TN) $.
\end{definition}
For different value of $ a_{i} $, $ B $ is became projective, conformal, concircular, quasi-conformal and conharmonic etc. curvature tensors (see \cite{shaikh2018some}). Therefore flatness of B-tensor also determine flatness of these tensors. Let's define $N$ is $ B-$flat if $ B=0 $. Then we have following. 

\begin{theorem}
	A complex Sasakian manifold can not be $ B-$flat.   
\end{theorem}
\begin{proof}
	Assume that $ B=0 $. Let chose $Y=Z=U$ and $ X=W=X\in \mathcal{H} $ then we have 
	\begin{eqnarray*}
		0 &=& a_{0}R(X,U,U,X) + a_{1}R(X,U,U,X) \\
		&+&a_{2}\rho(U,U)g(X,X)+a_{5}\rho(X,X)g(U,U)\\
		&+& \tau\{a_{8}g(X,X)g(U,U)\}.
	\end{eqnarray*}
	and from (\ref{SasakianR(XUU)}) and (\ref{Ricciler})  we get
\begin{align}
\rho(X,X)=-\frac{a_0+a_1+4pa_2+\tau a_8}{a_5} \label{B-tensörro(X,X)}.
\end{align}
	Therefore, since $ \rho(X,X) $ is constant $ \rho(X,X)=\rho(Y,Y) $. 
	On the other hand for unit and orthogonal horizontal vector fields $ X,Y $ by choosing   $Y=Z=Y \in \mathcal{H}$ and $ X=W=X\in \mathcal{H} $ under $B=0 $ condition we have 
	\begin{equation*}
	R(X,Y,Y,X)=-\frac{a_{2}\rho(Y,Y)+a_{5}\rho(X,X)+a_{8}\tau}{a_{0}+a_{1}}
	\end{equation*}
	and thus we get 
		\begin{align*}
	R(X,Y,Y,X)=\frac{a_{2}a_{0}+a_{1}a_{2}+a_{0}a_{5}+a_{1}a_{5}
		+4pa_{2}a_{2}^2+4pa_{2}a_{5}+\tau a_{8}a_{2}}{a_{5}(a_{0}+a_{1})}.
	\end{align*}
	This shows us the sectional curvature is independent of $ X $ and $ Y $ and then $ k(X,JX)=\mathcal{GH}(X) $. But from (\ref{k(X,JX)=GH(X)+3}) there is a contradiction. So $ N $ could not be $ B- $flat. 
\end{proof}
In \cite{turgut2017conformal} two of presented authors showed that a normal{\color{white}$ \theta $}complex{\color{white}$ \theta $}contact{\color{white}$ \theta $}manifold is not be conformal,concircular, quasi-conformal and conharmonic flat.
 The other notion of flatness is $ \phi-T $-flatness for a $(1,1)-  $tensor $ T $. Now we examine this notion for a complex Sasakian manifold. 
\begin{definition}
	Let $N$ be a complex Sasakian manifold. If 
	\begin{equation}
	G^{2}(\mathcal{B}(GX,GY)GZ)=0\text{ and }H^{2}(\mathcal{B}(HX,HY)HZ)=0
	\label{gh-conformal condition}
	\end{equation}%
	for $T,Y,Z$ vector fields on $N$ then $N$ is called a $\mathcal{GH}-$ $ \mathcal{B}-$flat.
\end{definition}
\begin{theorem}
	A complex Sasakian manifold can not be $\mathcal{GH}-$ $ \mathcal{B}-$flat.  
\end{theorem}
\begin{proof}
	Let $ N $ be a $ \mathcal{B}-$flat complex Sasakian manifold. For  $ X,Y, Z,W\in \Gamma(\mathcal{H}) $ we have 
	\begin{eqnarray*}
g(B(GX,GY)GZ, GW)&&=B(GX,GY,GZ,GW)\\
 &&= a_{0}R(GX,GY,GZ,GW) + a_{1}R(GX,GY,GZ,GW)\\
&&+ a_{2}\rho(GY,GZ)g(GX,GW)+a_{3}\rho(GX,GZ)g(GY,GW)\notag\\
&&+a_{4}\rho(GX,GY)g(GZ,GW)+a_{5}\rho(GX,GW)g(GY,GZ)\notag \\
&&+a_{6}\rho(GY,GW)g(GX,GZ)+a_{7}\rho(GZ,GW)g(GX,GY) \notag \\
&&+ \tau\{a_{8}g(GX,GW)g(GY,GZ)+a_{9}g(GX,GZ)g(GY,GW) \notag\\
&&+a_{10}g(GX,GY)g(GZ,GW)\}. \notag
\end{eqnarray*}
From (\ref{r(gx,gy,gz,gw)}) and (\ref{Ricciler}) we get 
\begin{eqnarray*}
g(B(GX,GY)GZ, GW)=B(GX,GY,GZ,GW)=B(X,Y,Z,W).
\end{eqnarray*}
Also we know from Theorem 4, $ B\neq 0 $, thus we have 
\begin{equation*}
g(B(GX,GY)GZ, GW)=-g(GB(GX,GY)GZ, W)\neq 0.
\end{equation*}
Because of $ W \in \Gamma(T\mathcal{H}) $ then $GB(GX,GY)GZ\neq0$ and so $G^2B(GX,GY)GZ\neq0$. 
\end{proof}
By this theorems we obtain that on a complex Sasakian manifold Weyl conformal curvature tensor, projective curvature tensor, concircular curvature tensor, conharmonic curvature tensor, quasi conformal curvature
tensor, pseudo projective curvature
tensor, quasi-concircular curvature tensor, pseudo
quasi conformal curvature tensor, M-projective curvature tensor, $ W_{i} $-curvature tensor, $ i = 1, 2, . . . ,9 $, $ W_{i}* $
-curvature, $ T $ -curvature tensor (see \cite{shaikh2018some} for details on tensors) can not vanish. In the other words a complex Sasakian manifold can not transform to flat space under any transformation such as conformal, projective, concircular etc.

\end{document}